\documentclass[11pt]{amsart}

\usepackage{amssymb, amsthm, amsmath, amsfonts, amsxtra, mathrsfs, mathtools}
\usepackage{color}
\usepackage{import}
\usepackage{thmtools}
\usepackage{thm-restate}
\usepackage[colorlinks, allcolors=blue]{hyperref}
\usepackage[capitalize]{cleveref}
\usepackage{chngcntr}
\usepackage[intoc]{nomencl}
\makenomenclature
\usepackage{stmaryrd}
\usepackage{svg}
\usepackage{tikz-cd}
\usepackage{tikz}
\usepackage{tikzsymbols}
\usetikzlibrary{decorations.pathreplacing,angles,quotes}
\usepackage{enumitem}
\usetikzlibrary{matrix}
\usepackage[ruled,vlined]{algorithm2e}
\usepackage[mathscr]{euscript}
\usepackage[normalem]{ulem}
\usepackage{comment}
\usepackage{multicol}
\usepackage{mathtools} 
\usepackage{caption, subcaption}

\input xy
\xyoption{all}

\usepackage[letterpaper, top=3cm, bottom=3cm,left=2.7cm, right=2.7cm, heightrounded,bindingoffset=0mm]{geometry}

\newtheorem{thm}{Theorem}[section]

\newtheorem{corollary}[thm]{Corollary}

\theoremstyle{definition}
\newtheorem{definition}[thm]{Definition}
\newtheorem{example}[thm]{Example}
\newtheorem{remark}[thm]{Remark}

\def\B{\mathbb{B}}

\def\G{\mathbb{G}}
\def\F{\mathbb{F}}
\def\R{\mathbb{R}}
\def\C{\mathbb{C}}
\def\P{\mathbb{P}}

\def\N{\mathbb{N}}

\def\calG{\mathcal{G}}

\def\calM{\mathcal{M}}
\def\M{\mathbb{M}}

\newcommand{\abs}[1]{\left\lvert#1\right\rvert}
\newcommand{\Aut}{\operatorname{Aut}}

\newcommand{\Span}{\operatorname{Span}}

\title{Equivariant log-concavity of graph matchings}
\author[S. Li]{Shiyue Li}\address{Department of Mathematics, Brown University, Providence, RI 02906}
\email{\url{shiyue_li@brown.edu}}
\date{\today}

\begin{document}
\maketitle
\begin{abstract}
    For any graph, we show that the graded permutation representation of the graph automorphism group given by matchings is strongly equivariantly log-concave. The proof gives a family of equivariant injections inspired by a combinatorial map of Kratthenthaler and reduces to the hard Lefschetz theorem. 
\end{abstract}

\section{Introduction}
Let $\G$ be a graph. A \textit{matching} on $\G$ is a set of pairwise non-adjacent edges. Let $\M_k$ be the set of matchings on $\G$ with $k$ edges. It is well-known that the sequence of graph matching numbers is \textit{log-concave}. That is, for positive integers ${k \le \ell}$, 
\[\abs{\M_{k-1}} \abs{\M_{\ell + 1}} \le \abs{\M_{k}} \abs{\M_{\ell}},\] The first proofs in \cite[Theorem 4.2]{Heilman1972Dimer} and \cite{Plummer1986matching} show that the generating polynomial of the graph matching sequence $\sum \abs{\M_k} x^k$ has only real roots, which implies log-concavity. Later, in \cite{Krattenthaler}, Kratthenthaler gave a combinatorial proof of this result. However, as we shall see in \cref{ex:six-cycle}, this   combinatorial approach breaks the graph symmetry. 

It is interesting to ask whether the graph matching sequence exhibits log-concavity that respects the graph symmetry. The notion of \textit{(strongly) equivariant log-concavity} arises when one considers log-concavity with respect to the symmetry of the underlying objects. It was introduced by Gedeon, Proudfoot and Young \cite[Section 5]{gedeon2017equivariant} in the context of the equivariant Kazhdan-Lusztig polynomial of a matroid. It is recently used to study other log-concave sequences in combinatorics, geometry and topology that involve group actions (see \cite{Proudfoot2018TheZO, matherne2021equivariant, gui2022equivariant, gui2022equivariantkahler}). 

We now recall precise definitions.   
\begin{definition}[{\cite[Section 5]{gedeon2017equivariant}}]
    Let $H$ be a group. 
    \begin{enumerate}
        \item A graded $H$-representation
        \[
        V^{\bullet} = \bigoplus_{k = 0} V_k
        \] is \textbf{$H$-equivariantly log-concave} if for all $i \ge 1$, there exists an $H$-equivariant injection 
        \[
        V_{i-1} \otimes V_{i + 1} \hookrightarrow V_i \otimes V_i. 
        \]
        \item A graded $H$-representation
        \[
        V^{\bullet} = \bigoplus_{k = 0} V_k
        \] is \textbf{strongly $H$-equivariantly log-concave} if for all positive integers $k \le \ell$, there exists an $H$-equivariant injection 
        \[
        V_{k-1} \otimes V_{\ell+1} \hookrightarrow V_{k} \otimes V_{\ell}. 
        \]
    \end{enumerate}
\end{definition}
Here, by an $H$-equivariant map, we mean the following: suppose two objects $X$ and $Y$ each afford an action by a group $H$. If a map $\varphi \colon X \to Y$ satisfies that 
\[
h \cdot \varphi(x) = \varphi(h \cdot x) 
\] for all $h$ in $H$ and $x$ in $X$, then $\varphi$ is $H$-equivariant.

In the case of graph matchings, the graded $\Aut(\G)$-representation is the graded permutation representation of $\Aut(\G)$ over any field $\F$ of characteristic $0$ \[V^{\bullet}_{\G} := \bigoplus_k \Span_{\F}\{ x_M \mid M \in \M_k\}.\] The aim of this paper is to prove the following theorem. 
\begin{thm}
\label{thm:main}
    The graded permutation representation $V_{\G}^{\bullet}$ of $\Aut(\G)$ given by matchings on $\G$  is strongly $\Aut(\G)$-equivariantly log-concave. 
\end{thm}
To prove this, we first construct an $\Aut(\G)$-equivariant map from $V_{k-1} \otimes V_{\ell + 1}$ to $V_{k} \otimes V_{\ell}$, and then decompose the tensor products $V_{k-1} \otimes V_{\ell+1}$ and $V_{k} \otimes V_{\ell}$ into direct sums with the following properties:
\begin{enumerate} 
    \item The direct sum decomposition of $V_{k-1} \otimes V_{\ell-1}$ is given by a partition on $\M_{k-1} \times \M_{\ell + 1}$. 
    \item The direct sum decomposition of $V_{k} \otimes V_{\ell}$ is given by a similar partition on $\M_{k} \times \M_{\ell}$. 
    \item On each direct summand $V(P) \subset V_{k-1} \otimes V_{\ell+1}$ associated with a part $P$, the map $\Phi_{k, \ell}$ is a hard Lefschetz operator into the direct summand $V(P') \subset V_{k} \otimes V_{\ell}$ determined by $P$. Injectivity follows from the hard Lefschetz theorem.
\end{enumerate} 
Note that taking dimensions of $V_{\G}^{\bullet}$ immediately recovers the non-equivariant log-concavity of graph matching numbers. 

As a corollary, if $\G$ is the graph of $n$ disjoint edges, then we obtain an $S_n$-equivariant categorification of the log-concavity of the binomial coefficients. 
\begin{corollary}
\label{cor}
    Given an integer $n \ge 1$, the graded permutation representation of $S_n$ given by 
    all $k$-subsets of $[n]$
    \[
    V^{\bullet} = \bigoplus_{k} \Span_{\F}\{x_{S} \mid S \subseteq [n], \abs{S} = k\}
    \] is $S_n$-equivariantly log-concave. 
\end{corollary}

\begin{remark}
    We situate both the object $V_{\G}^{\bullet}$ and the techniques in the present work in a slightly broader context. The assignment to each graph $\G$ of a graded $\Aut(\G)$-representation $V_{\G}^{\bullet}$ is also a representation of the \textit{graph minor category} $\calG$ (see \cite{miyata2020categorical}). It consists of graphs as objects, and compositions of edge deletions, contractions and automorphisms as morphisms. For each graph $\G$, one can also take the simplicial complex of matchings $\calM(\G)$ on $\G$ and study its homology $H(\calM(\G))$ (see \cite{miyata2020graph}). One connection is that, the differentials that define the homology $H(\calM(\G))$ also equip $V_{\G}^{\bullet}$ with a natural family of differentials that respects the graph minor morphisms. It would be very interesting to see further interactions between $V_{\G}^{\bullet}$ and other representations of the category $\calG$. 
\end{remark}

\section*{Acknowledgement} The author is grateful to Melody Chan for suggesting this interesting problem, and to Nick Proudfoot for many useful conversations. Thanks also go to two anonymous referees and Weiyan Chen, whose suggestions improved the paper and corrected small errors in an early draft. The author would also like to acknowledge support by the Coline M. Makepeace Fellowship from Brown University and the NSF grant DMS-1844768.

\section{Proof of \cref{thm:main}}
Let $k \le \ell$ be positive integers less than the maximum matching size on $\mathbb{G}$, and let $\M_{k, \ell} = \M_{k} \times \M_{\ell}$. 

\subsection{Construction}
\label{subsec:construction} Our strategy is to assign a subset of $\M_{k, \ell}$ to each pair $(M, M')$ in $\M_{k-1, \ell+1}$. The procedure goes as follows: 
\begin{enumerate}
    \item[(a)] Color $M$ with blue and $M'$ with pink, then each edge has $0, 1$, or $2$ colors. If we consider only the set of edges with exactly $1$ color, then we obtain a subgraph $\G_{M, M'}$. That is, $\G_{M, M'}$ is the symmetric difference of $M$ and $M'$. 
    \item[(b)] Since $M, M'$ are matchings, each vertex in $\G_{M, M'}$ cannot have degree $> 2$. That is, any connected component in $\G_{M, M'}$ is one of the two types:   
    \begin{enumerate}
        \item[(i)] a chain; or 
        \item[(ii)] a cycle with even number of edges. 
    \end{enumerate}
    Let $\C_{M, M'}$ denote the set of chains in $\G_{M, M'}$ with odd length. 
    By definition, any chain in $\C_{M, M'}$ must start and end with the same color. If a chain starts and ends with blue (respectively, pink) edges, we call it a blue (respectively, pink) chain. 
    We write 
    \begin{align*}
    \B_{M, M'} &:= \{\textrm{blue chains in } \C_{M, M'}\}, \\ 
    \P_{M, M'} &:= \{\textrm{pink chains in } \C_{M, M'}\}.
    \end{align*}  
    
    \item[(c)] Choose any pink chain in $\P_{M, M'}$. If we swap the edge colors (blue to pink, pink to blue), then we obtain an additional blue chain, and the number of blue (respectively, pink) edges in $\G_{M, M'}$ increases (respectively, decreases) by $1$. Therefore, if we record all the blue (respectively, pink) edges in this new edge-colored graph, we obtain a pair $(N, N')$ in $\M_{k, \ell}$. 
    \item[(d)] Do (c) for all pink chains in $\P_{M, M'}$. By construction, the resulting pairs are all distinct. Therefore, we obtain a subset of $\M_{k, \ell}$ and denote it as $\N_{M, M'}$. 
\end{enumerate}
We highlight some numerical facts. 
Before (c), the contributions of blue and pink edges in $\G_{M, M'}$ from every type of components are as follows.
    \begin{enumerate}
        \item[(i)] An even chain or cycle contributes the same number of blue and pink edges. 
        \item[(ii)] A blue chain contributes $1$ more blue edge than pink edges.
        \item[(iii)] A pink chain contributes $1$ more pink edge than blue edges.
    \end{enumerate}
    If an edge is not in $\G_{M, M'}$, then it must have no color or be both blue and pink. 
    Therefore we have that 
    \[
    \label{eq:r-minus-b}
    \abs{\P_{M, M'}} - \abs{\B_{M, M'}} = (\ell + 1) - (k - 1) = k - \ell + 2 \ge 2. 
    \]
    The positivity of this difference ensures that there exists a pink chain on which to perform (c).  
    This, together with the fact that $\abs{\B_{M, M'}} + \abs{\P_{M, M'}} = \abs{\C_{M, M'}}$, implies that 
    \[
    \abs{\B_{M, M'}} \le \frac{\abs{\C_{M, M'}}}{2} - 1.
    \]
After (d), we have $\abs{\N_{M, M'}} = \abs{\P_{M, M'}} \ge 2 + \abs{\B_{M, M'}} \ge 2$ by construction.

\begin{example}
Let $(M, M')$ a pair in $\M_{1, 3}$ in the $6$-cycle graph $\G$ as shown. The construction above maps the pair $(M, M')$ to the set $\N_{M, M'} = \{(N_1, N_1'), (N_2, N_2')\}$ in $\M_{2, 2}$ (\cref{fig:new-matchings}). 
\begin{figure}[h!]
    \centering
    \definecolor{pink}{rgb}{1, 0.4, 0.7}
\definecolor{blue}{rgb}{0.2, 0.8, 1}

%%%%%%%%%%%%%%%%%%%%%%%%%%%%%%% NEW FIGURE
\begin{tikzpicture}[style=ultra thick]
%%%%%%%%%%%%%%%%%%%%%%%%%%%%%%% M
\newdimen\R
\R=0.8cm
\draw  (0:\R)  -- (60:\R);
\draw (60:\R)  -- (120:\R);
\draw (120:\R)  -- (180:\R);
\draw (180:\R)  -- (240:\R);
\draw [blue] (240:\R)  -- (300:\R);
\draw (300:\R)  -- (360:\R);
\draw (360:\R)  -- (0:\R);

\foreach \x in {60,120,180,240,300,360}
\node[inner sep=1.2pt,circle,black,draw,fill] at (\x:\R) {};

\node at (0, -1.2) {$M$};
\node at (2, -1.2) {$M'$};

%%%%%%%%%%%%%%%%%%%%%%%%%%%%%%% M'
\begin{scope}[style=ultra thick,shift={(2,0)}]
\newdimen\R
\R=0.8cm
\R=0.8cm
\draw (0:\R)  -- (60:\R);
\draw [pink, dotted] (60:\R)  -- (120:\R);
\draw (120:\R)  -- (180:\R);
\draw [pink, dotted] (180:\R)  -- (240:\R);
\draw (240:\R)  -- (300:\R);
\draw [pink, dotted] (300:\R)  -- (360:\R);
\draw (360:\R)  -- (0:\R);
\foreach \x in {60,120,180,240,300,360}
\node[inner sep=1.2pt,circle,black,draw,fill] at (\x:\R) {};
\end{scope}

%%%%%%%%%%%%%%%%%%%%%%%%%%%%%%% An arrow
\begin{scope}[style=ultra thick]
\draw [ultra thick,->] (3.5,0) --(4.5,0);
\end{scope}
%%%%%%%%%%%%%%%%%%%%%%%%%%%%%%% N_1
\begin{scope}[style=ultra thick, shift={(6,1)}]
\newdimen\R
\R=0.8cm
\draw (0:\R)  -- (60:\R);
\draw [blue] (60:\R)  -- (120:\R);
\draw (120:\R)  -- (180:\R);
\draw (180:\R)  -- (240:\R);
\draw [blue] (240:\R)  -- (300:\R);
\draw (300:\R)  -- (360:\R);
\draw (360:\R)  -- (0:\R);

\foreach \x in {60,120,180,240,300,360}
\node[inner sep=1.2pt,circle,black,draw,fill] at (\x:\R) {};

\node at (0, -1.2) {$N_1$};
\node at (2, -1.2) {$N_1'$};
\end{scope}
%%%%%%%%%%%%%%%%%%%%%%%%%%%%%%% N_1'
\begin{scope}[style=ultra thick,shift={(8,1)}]
\newdimen\R
\R=0.8cm
\draw (0:\R)  -- (60:\R);
\draw (60:\R)  -- (120:\R);
\draw (120:\R)  -- (180:\R);
\draw [pink, dotted] (180:\R)  -- (240:\R);
\draw (240:\R)  -- (300:\R);
\draw [pink, dotted] (300:\R)  -- (360:\R);
\draw (360:\R)  -- (0:\R);
\foreach \x in {60,120,180,240,300,360}
\node[inner sep=1.2pt,circle,black,draw,fill] at (\x:\R) {};
\end{scope}

%%%%%%%%%%%%%%%%%%%%%%%%%%%%%%% N_1
\begin{scope}[style=ultra thick, shift={(6,-1.5)}]
\newdimen\R
\R=0.8cm
\draw (0:\R)  -- (60:\R);
\draw (60:\R)  -- (120:\R);
\draw (120:\R)  -- (180:\R);
\draw [blue] (180:\R)  -- (240:\R);
\draw (240:\R)  -- (300:\R);
\draw [blue] (300:\R)  -- (360:\R);
\draw (360:\R)  -- (0:\R);

\foreach \x in {60,120,180,240,300,360}
\node[inner sep=1.2pt,circle,black,draw,fill] at (\x:\R) {};

\node at (0, -1.3) {$N_2$};
\node at (2, -1.3) {$N_2'$};
\end{scope}
%%%%%%%%%%%%%%%%%%%%%%%%%%%%%%% N_1'
\begin{scope}[style=ultra thick,shift={(8,-1.5)}]
\newdimen\R
\R=0.8cm
\draw (0:\R)  -- (60:\R);
\draw [pink, dotted] (60:\R)  -- (120:\R);
\draw (120:\R)  -- (180:\R);
\draw (180:\R)  -- (240:\R);
\draw [pink, dotted]  (240:\R)  -- (300:\R);
\draw (300:\R)  -- (360:\R);
\draw (360:\R)  -- (0:\R);
\foreach \x in {60,120,180,240,300,360}
\node[inner sep=1.2pt,circle,black,draw,fill] at (\x:\R) {};
\end{scope}

\end{tikzpicture}
    \caption{}
    \label{fig:new-matchings}
\end{figure}
\end{example}

Now we define the linear map $\Phi_{k, \ell} \colon V_{k - 1} \otimes V_{\ell + 1} \to V_{k} \otimes V_{\ell}$ by 
\[
    x_{M} \otimes x_{M'} \mapsto \frac{1}{\abs{\N_{M, M'}}} \sum_{(N, N')} x_{N} \otimes x_{N'}, 
\] where $(N, N')$ ranges over all pairs in $\N_{M, M'}$ and extend linearly.

\subsection{Proof of equivariance}

This section is devoted to show that $\Phi_{k, \ell}$ is $\Aut(\G)$-equivariant. That is, for any $\sigma$ in $\Aut(\G)$ and $(M, M')$ in $\M_{k-1, \ell+1}$, 
\[
    \Phi_{k, \ell}( \sigma \cdot (x_{M} \otimes x_{M'})) = \sigma \cdot \Phi_{k, \ell}(x_M \otimes x_{M'}).
\]

The left hand side is 
\[
\Phi_{k, \ell}( \sigma \cdot (x_{M} \otimes x_{M'})) = \Phi_{k, \ell}( x_{\sigma(M)} \otimes  x_{\sigma(M')}) 
        = \frac{1}{\abs{\N_{\sigma(M), \sigma(M')}}} \sum_{(N, N')} x_{N} \otimes x_{N'},
\] where $(N, N')$ ranges over all pairs in $\N_{\sigma(M), \sigma(M')}$. 

The right hand side is 
\[
\sigma \cdot \Phi_{k, \ell}(x_M \otimes x_{M'}) = \sigma \cdot \bigg(\frac{1}{\abs{\N_{M, M'}}} \sum_{(N, N')} x_N \otimes x_{N'} \bigg)
        = \frac{1}{\abs{\N_{M, M'}}} \sum_{(N, N')}  x_{\sigma(N)} \otimes x_{\sigma(N')}, 
\]
where $(N, N')$ ranges over all pairs in $\N_{M, M'}$. 

Crucially, producing the set $\N_{M, M'}$ and performing $\sigma$ on the pair $(M, M')$ are commutative. In other words, the set equality 
\[
    \N_{\sigma(M), \sigma(M')} = \sigma(\N_{M, M'}) = \{(\sigma(N), \sigma(N')) \colon (N, N') \in \N_{M, M'}\}
\] holds. Moreover, by virtue of $\sigma$ being a graph automorphism, the numerical equality
\[
\frac{1}{\abs{\N_{M, M'}}} = \frac{1}{\abs{\N_{\sigma(M), \sigma(M')}}}
\] holds. 
These two identities imply that the two sides in the desired equation coincide. 

\subsection{Proof of injectivity} 
\label{subsec:injectivity}
This section is devoted to show that $\Phi_{k, \ell}$ is injective. 
As stated in the introduction, we do so by decomposing the tensor products $V_{k-1} \otimes V_{\ell + 1}$ and $V_{k} \otimes V_{\ell}$ into direct sums with the following properties:
\begin{enumerate} 
    \item The direct sum decomposition of $V_{k-1} \otimes V_{\ell-1}$ is given by a partition on $\M_{k-1} \times \M_{\ell + 1}$. 
    \item The direct sum decomposition of $V_{k} \otimes V_{\ell}$ is given by a similar partition on $\M_{k} \times \M_{\ell}$. 
    \item On each direct summand $V(P) \subset V_{k-1} \otimes V_{\ell+1}$ associated with a part $P$, the $\Phi_{k, \ell}$ is a hard Lefschetz operator into the direct summand $V(P') \subset V_{k} \otimes V_{\ell}$ determined by $P$. Injectivity follows from the hard Lefschetz theorem.
\end{enumerate}  
    
First, we create a partition $\Pi_{k-1, \ell+1}$ on $\M_{k-1, \ell+1}$ by giving an equivalence relation: any two pairs $(M_1, M_2)$ and $(M'_1, M'_2)$ in $\M_{k-1, \ell+1}$ are equivalent if they form the same subgraph and yield the same coloring outside of odd chains with $1$-colored edges. More precisely, 
\begin{enumerate}
    \item[(i)] $M_1 \cup M_1' = M_2 \cup M_2'$; 
    \item[(ii)] $M_1 = M_2$, $M_1' = M_2'$ outside of $\C_{M_1, M_1'}$. Note that, $\C_{M_1, M_1'} = \C_{M_2, M_2'}$ by (i). The restriction thus makes sense.
\end{enumerate} 

Any part $P$ of the partition $\Pi_{k-1, \ell +1}$ has the following properties: for any $(M_1, M_1')$ and $(M_2, M_2')$ in $P$,
\begin{enumerate}
    \item the sets of chains $\C_{M_1, M_1'}$ and $\C_{M_2, M_2'}$ are equal as subgraphs; and 
    \item the cardinalities of the blue chains are equal, i.e., 
    \[
    \abs{\B_{M_1, M_1'}} = \abs{\B_{M_2, M_2'}}.
    \] 
\end{enumerate}
Indeed, (1) follows from the definition of $P$, and (2) follows for numerical reasons:  
\[
\abs{\B_{M_1, M_1'}} + \abs{\P_{M_1, M_1'}} = \abs{\C_{M_1, M_1'}} = \abs{\C_{M_2, M_2'}} =  \abs{\B_{M_2, M_2'}} + \abs{\P_{M_2, M_2'}}
\] and 
\[
 \abs{\P_{M_1, M_1'}} - \abs{\B_{M_1, M_1'}} = (\ell +1) - (k-1) = \abs{\P_{M_2, M_2'}} - \abs{\B_{M_2, M_2'}}. 
\] Thus we have that 
\[
\abs{\B_{M_1, M_1'}} = \abs{\B_{M_2, M_2'}}. 
\] Therefore, we set $\C_{P} := \C_{M, M'}$  and $\abs{\B_{P}} := \abs{\B_{M, M'}}$ for any $(M, M')$ in $P$. 

Now, consider the following vector space for any part $P$ in $\Pi_{k-1, \ell+1}$ 
\[
    V_{k-1, \ell+1}(P) := \Span_{\F}\{x_M \otimes x_M \mid (M, M') \in P\}. 
\]
We now realize $V_{k-1, \ell+1}(P)$ as a categorification of the $\abs{\B_{P}}$th level of the Boolean lattice on $\C_P$.
Consider the map 
\[
\beta_P \colon P \to \binom{\C_P}{\abs{\B_{P}}}, \quad (M, M') \mapsto \B_{M, M'}.
\]\\
It is well-defined by the construction of $\B_{M, M'}$ in step (b) of \cref{subsec:construction}. 
It is surjective: for each subset of edges with size $\abs{\B_P}$ of $\C_{P}$, one can color the present subset as blue chains and reverse the last step in (c). It is injective: if two pairs $(M_1, M_1')$ and $(M_2, M_2')$ in $P$ give the same set of blue chains, then the fact 
\[
\B_{M_1, M_1'} \sqcup \P_{M_2, M_2'} = \C_P = \B_{M_2, M_2'} \sqcup \P_{M_2, M_2'}
\] implies that $\P_{M_1, M_1'} = \P_{M_2, M_2'}$. This further means $M_1 = M_2, M_1' = M_2'$ on $\C_{P}$, and moreover $(M_1, M_1') = (M_2, M_2')$. Therefore, $\beta_P$ is a bijection. 

Next, we consider the vector space 
\[
V_{\C_P, \abs{\B_P}} := \Span_{\F} \bigg\{y_{\B} \mid \B \in \binom{\C_P}{\abs{\B_P}}\bigg\}, 
\] and define
\[
\underline{\beta_P} \colon V_{k-1, \ell+1}(P) \to V_{\C_P, \abs{\B_{P}}}, \quad x_{M} \otimes x_{M'} \mapsto y_{\B_{M, M'}}.
\]
It is an isomorphism of vector spaces, because $\beta_P$ is a bijection on the bases.

Then, we do the same procedure for $\M_{k, \ell}$. If we create a partition $\Pi_{k, \ell}$ on $\M_{k, \ell}$ in the same way, then the set 
\[
P' = \bigcup_{(M, M') \in P} \N_{M, M'} 
\] for $P$ in $\Pi_{k-1, \ell+1}$ is also a part in $\Pi_{k, \ell}$. 
Indeed, we need to show that $P'$ satisfies the following properties.  
\begin{enumerate}
    \item Elements in $P'$ are pairwise equivalent. Indeed, since the procedure (d) does not change edges outside of $\C_{M, M'}$, any $(N_1, N_1')$ and $(N_2, N_2')$ in $P'$ must satisfy that $N_1 \cup N_1' = N_2 \cup N_2'$ and $N_1 = N_2, N_1' = N_2'$ outside of $\C_{M, M'}$.
    \item Any pair that is equivalent to a member in $P'$ must also belong to $P$. That is, if a pair $(M, M')$ in $\M_{k, \ell}$ is equivalent to a pair $(N, N')$ in $P'$, then there exists $(\underline{M}, \underline{M}')$ in $P$ such that 
    \[(M, M') \in \N_{\underline{M}, \underline{M}'}.\] 
    Indeed, such $(\underline{M}, \underline{M}')$ can be constructed by reversing the procedure (c) using any pink chain in $\P_{M, M'}$. 
\end{enumerate} We define the map $\beta_{P'}$ and $\underline{\beta_{P'}}$ similar to those for $P$. Note that, by construction, \[\abs{\B_{P'}} = \abs{\B_P} + 1 \quad \textrm{and} \quad \C_{P'} = \C_P.\]

Finally, for each $P$ in $\Pi_{k-1, \ell+1}$, define the linear map 
\[
L_{P} \colon V_{\C_P, \abs{\B_{P}}} \to V_{\C_{P'}, \abs{\B_P}+1}, \quad y_{\B} \mapsto \frac{1}{\abs{\C_P} - \abs{\B_P}} \sum_{\B \subseteq \B' \in \binom{\C_P}{\abs{\B_P} + 1}} y_{\B'}.
\]
Crucially, $L_P$ is the hard Leftschetz operator on the graded vector space spanned by all subsets of $\C_P$, where the grading is given by cardinality. It is injective for degrees $\abs{\B_P} \le \abs{\C_P}/2 - 1$. This operator and its injectivity on the lower half graded pieces have been studied in various contexts. We invite the reader to see proofs of various flavors: \cite{stanley1980weyl}, \cite[The hard Lefschetz theorem]{stanley1983combinatorial}, \cite[Proposition 7]{hara2008determinants}, \cite{harima2013lefschetz}, \cite[Theorem 4.7]{stanley2013algebraic} and \cite[{Theorem 1.1(3)}]{braden2020singular}. 

By construction, the following diagram commutes: 
\begin{center}
\begin{tikzcd}
    V_{k-1, \ell+1}(P) \ar{r}{\beta_P} \ar[swap]{r}{\cong}\ar{d}{\Phi_{k, \ell}} & V_{\C_P, \abs{\B_P}} \ar{d}{L_P} \\
    V_{k, \ell}(P') \ar{r}{\beta_{P'}} \ar[swap]{r}{\cong}& V_{\C_{P'}, \abs{\B_P} + 1}.
\end{tikzcd}
\end{center}
Therefore, $\Phi_{k, \ell}$ is injective from $V_{k-1, \ell+1}(P)$ to $V_{k, \ell}(P')$.

Note that by construction, 
\[
    V_{k-1} \otimes V_{\ell+1} = \bigoplus_{P \in \Pi_{k-1, \ell+1}} V_{k-1, \ell+1}(P) \cong \bigoplus_{P \in \Pi_{k-1, \ell+1}} V_{\C_P, \abs{\B_P}}.
\]  
Then the last sentence of the previous paragraph implies that $\Phi_{k, \ell}$ is injective on $V_{k-1} \otimes V_{\ell+1}$.

\begin{remark}
Our construction was inspired by Krattenthaler's combinatorial proof of the non-equivariant log-concavity of graph matchings in \cite{Krattenthaler}. His proof constructs an injective set map
\[f_{k, \ell} \colon \M_{k-1, \ell+1} \to \M_{k, \ell}.\] This injective map depends on a vertex order on the graph to select only $1$ pink chain to convert into a blue chain; for one concrete example, see \cite[Section 2.2]{stanton1986constructive}. 
In general, this set map is not $\Aut(\G)$-equivariant. See the following example. Our approach fixes where the map $f_{k, \ell}$ breaks the graph symmetry. Interestingly, this fix also reduces the proof of the injectivity to the hard Lefschetz theorem, as mentioned before. 
\end{remark}
\begin{example}
\label{ex:six-cycle}
    Let $(M, M')$ be a pair in $\M_{1, 3}$ in the $6$-cycle vertex-ordered graph $\G$ as shown. When converting a pink chain to a blue chain, we choose the pink chain containing the minimum vertex. (We invite the reader to check that this is indeed the injective map given in \cite[Algorithm 14, Section 2.2]{stanton1986constructive}.)  Together with the automorphism $\rho$ given by the clockwise $2\pi/6$ rotation, this pair is a witness of non-equivariance of $f = f_{k, \ell}$, i.e., \[\rho(f(M)), \rho(f(M')) \ne f(\rho(M)), f(\rho(M')).\]
    \begin{figure}[h!]
        \centering
        \definecolor{pink}{rgb}{1, 0.4, 0.7}
\definecolor{blue}{rgb}{0.2, 0.8, 1}

%%%%%%%%%%%%%%%%%%%%%%%%%%%%%%% NEW FIGURE
\begin{tikzpicture}[style=ultra thick, scale=0.8]
%%%%%%%%%%%%%%%%%%%%%%%%%%%%%%% M

\begin{scope}[style=ultra thick,shift={(-5,0)}]
\newdimen\R
\R=0.8cm
\draw  (0:\R)  -- (60:\R);
\draw (60:\R)  -- (120:\R);
\draw (120:\R)  -- (180:\R);
\draw (180:\R)  -- (240:\R);
\draw [blue] (240:\R)  -- (300:\R);
\draw (300:\R)  -- (360:\R);
\draw (360:\R)  -- (0:\R);
\node[scale=0.5] at (120:\R-0.2cm) {$1$};
\node[scale=0.5] at (60:\R-0.2cm) {$2$};
\node[scale=0.5] at (0:\R-0.2cm) {$3$};
\node[scale=0.5] at (300:\R-0.2cm) {$4$};
\node[scale=0.5] at (240:\R-0.2cm) {$5$};
\node[scale=0.5] at (180:\R-0.2cm) {$6$};

\foreach \x in {60,120,180,240,300,360}
\node[inner sep=1.2pt,circle,black,draw,fill] at (\x:\R) {};

\node at (0.0, -1.2) {$M$};
\node at (2.0, -1.2) {$M'$};
\end{scope}
%%%%%%%%%%%%%%%%%%%%%%%%%%%%%%% M'
\begin{scope}[style=ultra thick,shift={(-3,0)}]
\newdimen\R
\R=0.8cm
\R=0.8cm
\draw (0:\R)  -- (60:\R);
\draw [pink, dotted] (60:\R)  -- (120:\R);
\draw (120:\R)  -- (180:\R);
\draw [pink, dotted] (180:\R)  -- (240:\R);
\draw (240:\R)  -- (300:\R);
\draw [pink, dotted] (300:\R)  -- (360:\R);
\draw (360:\R)  -- (0:\R);
\node[scale=0.5] at (120:\R-0.2cm) {$1$};
\node[scale=0.5] at (60:\R-0.2cm) {$2$};
\node[scale=0.5] at (0:\R-0.2cm) {$3$};
\node[scale=0.5] at (300:\R-0.2cm) {$4$};
\node[scale=0.5] at (240:\R-0.2cm) {$5$};
\node[scale=0.5] at (180:\R-0.2cm) {$6$};

\foreach \x in {60,120,180,240,300,360}
\node[inner sep=1.2pt,circle,black,draw,fill] at (\x:\R) {};
\end{scope}

%%%%%%%%%%%%%%%%%%%%%%%%%%%%%%% First upper level right arrow
\begin{scope}
\draw [thick,->] (-1.8,0.5) --(-0.8,1.2);
\node at (-1.5,1) {$f$};
\end{scope}

%%%%%%%%%%%%%%%%%%%%%%%%%%%%%%% First lower level right arrow
\begin{scope}
\draw [thick,->] (-1.8,-0.5) --(-0.8,-1.2);
\node at (-1.5,-1) {$\rho$};
\end{scope}
%%%%%%%%%%%%%%%%%%%%%%%%%%%%%%% Second lower level right arrow
\begin{scope}
\draw [thick,->] (3.8,-1.5) --(4.8,-1.5);
\node at (4.3,-1.8) {$f$};
\end{scope}
%%%%%%%%%%%%%%%%%%%%%%%%%%%%%%% Second right vertical arrow
\begin{scope}
\draw [thick,->] (3.8,1.3) --(4.8,1.3);
\node at (4.3,1.8) {$\rho$};
\end{scope}

%%%%%%%%%%%%%%%%%%%%%%%%%%%%%%% rotated M_1
\begin{scope}[style=ultra thick, shift={(0.5,-1.5)}]
\newdimen\R
\R=0.8cm
\draw (0:\R)  -- (60:\R);
\draw (60:\R)  -- (120:\R);
\draw (120:\R)  -- (180:\R);
\draw [blue] (180:\R)  -- (240:\R);
\draw (240:\R)  -- (300:\R);
\draw (300:\R)  -- (360:\R);
\draw (360:\R)  -- (0:\R);

\node[scale=0.5] at (120:\R-0.2cm) {$1$};
\node[scale=0.5] at (60:\R-0.2cm) {$2$};
\node[scale=0.5] at (0:\R-0.2cm) {$3$};
\node[scale=0.5] at (300:\R-0.2cm) {$4$};
\node[scale=0.5] at (240:\R-0.2cm) {$5$};
\node[scale=0.5] at (180:\R-0.2cm) {$6$};

\foreach \x in {60,120,180,240,300,360}
\node[inner sep=1.2pt,circle,black,draw,fill] at (\x:\R) {};

\node at (0, -1.2) {$\rho(M)$};
\node at (2, -1.2) {$\rho(M')$};
\end{scope}
%%%%%%%%%%%%%%%%%%%%%%%%%%%%%%%  rotated M_1' 
\begin{scope}[style=ultra thick,shift={(2.5,-1.5)}]
\newdimen\R
\R=0.8cm
\draw [pink, dotted]  (0:\R)  -- (60:\R);
\draw (60:\R)  -- (120:\R);
\draw [pink, dotted]  (120:\R)  -- (180:\R);
\draw (180:\R)  -- (240:\R);
\draw [pink, dotted] (240:\R)  -- (300:\R);
\draw (300:\R)  -- (360:\R);
\draw (360:\R)  -- (0:\R);

\node[scale=0.5] at (120:\R-0.2cm) {$1$};
\node[scale=0.5] at (60:\R-0.2cm) {$2$};
\node[scale=0.5] at (0:\R-0.2cm) {$3$};
\node[scale=0.5] at (300:\R-0.2cm) {$4$};
\node[scale=0.5] at (240:\R-0.2cm) {$5$};
\node[scale=0.5] at (180:\R-0.2cm) {$6$};

\foreach \x in {60,120,180,240,300,360}
\node[inner sep=1.2pt,circle,black,draw,fill] at (\x:\R) {};
\end{scope}
%%%%%%%%%%%%%%%%%%%%%%%%%%%%%%% f of M_1
\begin{scope}[style=ultra thick, shift={(0.5,1.3)}]
\newdimen\R
\R=0.8cm
\draw (0:\R)  -- (60:\R);
\draw [blue] (60:\R)  -- (120:\R);
\draw (120:\R)  -- (180:\R);
\draw (180:\R)  -- (240:\R);
\draw [blue] (240:\R)  -- (300:\R);
\draw (300:\R)  -- (360:\R);
\draw (360:\R)  -- (0:\R);

\node[scale=0.5] at (120:\R-0.2cm) {$1$};
\node[scale=0.5] at (60:\R-0.2cm) {$2$};
\node[scale=0.5] at (0:\R-0.2cm) {$3$};
\node[scale=0.5] at (300:\R-0.2cm) {$4$};
\node[scale=0.5] at (240:\R-0.2cm) {$5$};
\node[scale=0.5] at (180:\R-0.2cm) {$6$};

\foreach \x in {60,120,180,240,300,360}
\node[inner sep=1.2pt,circle,black,draw,fill] at (\x:\R) {};

\node at (0, -1.3) {$f(M)$};
\node at (2, -1.3) {$f(M')$};
\end{scope}
%%%%%%%%%%%%%%%%%%%%%%%%%%%%%%% f of M_1'
\begin{scope}[style=ultra thick,shift={(2.5,1.3)}]
\newdimen\R
\R=0.8cm
\draw (0:\R)  -- (60:\R);
\draw (60:\R)  -- (120:\R);
\draw (120:\R)  -- (180:\R);
\draw [pink, dotted] (180:\R)  -- (240:\R);
\draw (240:\R)  -- (300:\R);
\draw [pink, dotted] (300:\R)  -- (360:\R);
\draw (360:\R)  -- (0:\R);

\node[scale=0.5] at (120:\R-0.2cm) {$1$};
\node[scale=0.5] at (60:\R-0.2cm) {$2$};
\node[scale=0.5] at (0:\R-0.2cm) {$3$};
\node[scale=0.5] at (300:\R-0.2cm) {$4$};
\node[scale=0.5] at (240:\R-0.2cm) {$5$};
\node[scale=0.5] at (180:\R-0.2cm) {$6$};

\foreach \x in {60,120,180,240,300,360}
\node[inner sep=1.2pt,circle,black,draw,fill] at (\x:\R) {};
\end{scope}

%%%%%%%%%%%%%%%%%%%%%%%%%%%%%%% f of rotated M_1
\begin{scope}[style=ultra thick, shift={(6,-1.5)}]
\newdimen\R
\R=0.8cm
\draw (0:\R)  -- (60:\R);
\draw (60:\R)  -- (120:\R);
\draw [blue] (120:\R)  -- (180:\R);
\draw (180:\R)  -- (240:\R);
\draw [blue] (240:\R)  -- (300:\R);
\draw (300:\R)  -- (360:\R);
\draw (360:\R)  -- (0:\R);

\node[scale=0.5] at (120:\R-0.2cm) {$1$};
\node[scale=0.5] at (60:\R-0.2cm) {$2$};
\node[scale=0.5] at (0:\R-0.2cm) {$3$};
\node[scale=0.5] at (300:\R-0.2cm) {$4$};
\node[scale=0.5] at (240:\R-0.2cm) {$5$};
\node[scale=0.5] at (180:\R-0.2cm) {$6$};

\foreach \x in {60,120,180,240,300,360}
\node[inner sep=1.2pt,circle,black,draw,fill] at (\x:\R) {};

\node at (0, -1.3) {$f(\rho(M))$};
\node at (2, -1.3) {$f(\rho(M'))$};
\end{scope}
%%%%%%%%%%%%%%%%%%%%%%%%%%%%%%% f of rotated M_2'
\begin{scope}[style=ultra thick,shift={(8,-1.5)}]
\newdimen\R
\R=0.8cm
\draw [pink, dotted] (0:\R)  -- (60:\R);
\draw (60:\R)  -- (120:\R);
\draw (120:\R)  -- (180:\R);
\draw [pink, dotted] (180:\R)  -- (240:\R);
\draw (240:\R)  -- (300:\R);
\draw (300:\R)  -- (360:\R);
\draw (360:\R)  -- (0:\R);

\node[scale=0.5] at (120:\R-0.2cm) {$1$};
\node[scale=0.5] at (60:\R-0.2cm) {$2$};
\node[scale=0.5] at (0:\R-0.2cm) {$3$};
\node[scale=0.5] at (300:\R-0.2cm) {$4$};
\node[scale=0.5] at (240:\R-0.2cm) {$5$};
\node[scale=0.5] at (180:\R-0.2cm) {$6$};
\foreach \x in {60,120,180,240,300,360}
\node[inner sep=1.2pt,circle,black,draw,fill] at (\x:\R) {};
\end{scope}

%%%%%%%%%%%%%%%%%%%%%%%%%%%%%%% rotated f of M_1
\begin{scope}[style=ultra thick, shift={(6,1.3)}]
\newdimen\R
\R=0.8cm
\draw [blue] (0:\R)  -- (60:\R);
\draw (60:\R)  -- (120:\R);
\draw (120:\R)  -- (180:\R);
\draw [blue] (180:\R)  -- (240:\R);
\draw (240:\R)  -- (300:\R);
\draw (300:\R)  -- (360:\R);
\draw (360:\R)  -- (0:\R);
\node[scale=0.5] at (120:\R-0.2cm) {$1$};
\node[scale=0.5] at (60:\R-0.2cm) {$2$};
\node[scale=0.5] at (0:\R-0.2cm) {$3$};
\node[scale=0.5] at (300:\R-0.2cm) {$4$};
\node[scale=0.5] at (240:\R-0.2cm) {$5$};
\node[scale=0.5] at (180:\R-0.2cm) {$6$};

\foreach \x in {60,120,180,240,300,360}
\node[inner sep=1.2pt,circle,black,draw,fill] at (\x:\R) {};

\node at (0, -1.3) {$\rho(f(M))$};
\node at (2, -1.3) {$\rho(f(M'))$};
\end{scope}
%%%%%%%%%%%%%%%%%%%%%%%%%%%%%%% rotated f of M_2'
\begin{scope}[style=ultra thick,shift={(8,1.3)}]
\newdimen\R
\R=0.8cm
\draw (0:\R)  -- (60:\R);
\draw (60:\R)  -- (120:\R);
\draw [pink, dotted] (120:\R)  -- (180:\R);
\draw (180:\R)  -- (240:\R);
\draw [pink, dotted] (240:\R)  -- (300:\R);
\draw (300:\R)  -- (360:\R);
\draw (360:\R)  -- (0:\R);
\node[scale=0.5] at (120:\R-0.2cm) {$1$};
\node[scale=0.5] at (60:\R-0.2cm) {$2$};
\node[scale=0.5] at (0:\R-0.2cm) {$3$};
\node[scale=0.5] at (300:\R-0.2cm) {$4$};
\node[scale=0.5] at (240:\R-0.2cm) {$5$};
\node[scale=0.5] at (180:\R-0.2cm) {$6$};
\foreach \x in {60,120,180,240,300,360}
\node[inner sep=1.2pt,circle,black,draw,fill] at (\x:\R) {};
\end{scope}

\end{tikzpicture}
        \caption{}
        \label{fig:kratthenthaler-rotated}
    \end{figure} 
\end{example}

\section{A weighted equivariant promotion}
Thanks to the injections, we also obtain an $\Aut(\G)$-equivariant promotion of the following theorem.  
\begin{thm}[{\cite[Theorem 2]{Krattenthaler}}]
    Fix a graph $\G$, the $k$th weighted matching number of $\G$ is a polynomial 
    \[
    \M_k(\G, \mathbf{x}) := \sum_{M \in \M_k} \prod_{e \in M} x_e \in F[x_e]_{e \in E(\G)}.
    \] 
    Then for positive integers ${k \le \ell}$ less than the maximum matching size, the polynomial 
    \[
    \M_{k}(\G, \mathbf{x}) \M_{\ell}(\G, \mathbf{x}) - \M_{k-1}(\G, \mathbf{x}) \M_{\ell + 1}(\G, \mathbf{x}) 
    \] has nonnegative coefficients. 
\end{thm}
The polynomial ring in variables $\{x_e\}_{e \in E(\G)}$ is naturally a graded $\Aut(\G)$-representation. The $\Aut(\G)$-equivariant promotion of the theorem above can be described as follows. For positive integers ${k \le \ell}$, define 
\[
\iota_{k,\ell} \colon V_{k} \otimes V_{\ell} \to F[x_e]_{e \in E(\G)}, \quad 
x_M \otimes x_{M'} \mapsto \left(\prod_{e \in M} x_e \right)\left(\prod_{e' \in M'} x_{e'} \right). 
\] By definition, $\iota_{k, \ell}$ is $\Aut(\G)$-equivariant. The following theorem immediately follows from construction. Note that the weighting in the definition of $\Phi_{k, \ell}$ is necessary for the commutativity. 
\begin{thm}
    For any graph $\G$ and positive integers $k \le \ell$, the following diagram of $\Aut(\G)$-equivariant maps commutes: 
    \begin{center}
    \begin{tikzcd}[row sep=large, column sep=large]
        V_{k -1} \otimes V_{\ell+1} \ar[r, "\iota_{k-1, \ell+1}"] \ar[d,hookrightarrow, "\Phi_{k, \ell}", swap] & F[x_e]_{e \in E(\G)} \\
        V_{k} \otimes V_{\ell} \ar[ur, "\iota_{k, \ell}", swap]
    \end{tikzcd}. 
    \end{center}
\end{thm}

\bibliographystyle{alpha}
\bibliography{bibliography.bib}
\end{document}